\documentclass[reqno]{amsart}

\usepackage{enumerate}
\usepackage{amsfonts,amssymb,amsmath,amsthm}
\usepackage{epsfig}
\usepackage{graphics}
\usepackage[normalem]{ulem}
\usepackage{color}
\usepackage{version}

\input xy 
\xyoption{all}
\numberwithin{equation}{section}

\definecolor{OrangeRed}{cmyk}{0,0.6,1,0}            
\definecolor{DarkBlue}{cmyk}{1,1,0,0.20}
\definecolor{DarkGreen}{cmyk}{1,0,0.6,0.2}
\definecolor{myblue}{rgb}{0.66,0.78,1.00}
\definecolor{Violet}{cmyk}{0.79,0.88,0,0}
\definecolor{Lavender}{cmyk}{0,0.48,0,0}

\newtheorem{thm}{Theorem}[section]
\newtheorem{theorem}[thm]{Theorem}
\newtheorem{main theorem}[thm]{Main Theorem}
\newtheorem{corollary}[thm]{Corollary}

\newtheorem{lemma}[thm]{Lemma}
\newtheorem{lem}[thm]{Lemma}
\newtheorem{prop}[thm]{Proposition}

\theoremstyle{definition}

\newtheorem{defn}[thm]{Definition}

\newtheorem{rem}[thm]{Remark}

\def\C{\mathbb C}
\def\P{\mathbb P}

\def\B{\mathbb B}
\def\bcases{\begin{cases}}

\def\ecases{\end{cases}}

\newcommand{\D}{\mathbb D}

\renewcommand{\Im}{\operatorname{Im}\,}

\newcommand{\N}{\mathbb N}

\newcommand{\R}{\mathbb R}

\newcommand{\Z}{\mathbb Z}

\newcommand{\bea}{\begin{eqnarray*}}
\newcommand{\eea}{\end{eqnarray*}}

\newcommand{\be}{\begin{equation}}
\newcommand{\ee}{\end{equation}}

\newcommand{\ra}{\rightarrow}

\renewcommand{\epsilon}{\varepsilon}
\renewcommand{\phi}{\varphi}
\newcommand{\U}{\mathbb{U}}

\newcommand{\Hh}{\hat{H}}

\begin{document}

\title{Dynamics of transcendental H\'enon maps-II}

\author[L. Arosio]{Leandro Arosio$^{\dag}$}
\author[A.M. Benini]{Anna Miriam Benini$^{\ddag}$}
\author[J.E.  Forn{\ae}ss ]{John Erik Forn{\ae}ss}
\author[H. Peters]{Han Peters}

\today
\thanks{$^{\dag}$  Supported by the SIR grant ``NEWHOLITE - New methods in holomorphic iteration'' no. RBSI14CFME. Partially supported by the MIUR Excellence Department Project awarded to the Department of Mathematics, University of Rome Tor Vergata,  CUP E83C18000100006.}
\thanks{$^{\ddag}$ This project has received funding from the European Union's Horizon 2020 research and innovation programme under the Marie Sk\l odowska-Curie Grant Agreement No. 703269   COTRADY}
\thanks{Part of this work was done during the international research program "Several Complex Variables and Complex Dynamics"
at the Center for Advanced Study at the Academy of Science and Letters in Oslo during the academic year 2016/2017. }
\address{ L. Arosio: Dipartimento Di Matematica\\
Universit\`{a} di Roma \textquotedblleft Tor Vergata\textquotedblright\  \\
 Italy} \email{arosio@mat.uniroma2.it}
\address{ A.M. Benini: Dipartimento  di   Matematica Fisica e Informatica\\
Universit\'a di Parma, IT.  \\
} \email{ambenini@gmail.com}
\address{ H. Peters: Korteweg de Vries Institute for Mathematics\\
University of Amsterdam\\
the Netherlands} \email{hanpeters77@gmail.com}
\address{ J.E. Fornaess: Department of Mathematical Sciences\\
NTNU Trondheim, Norway} \email{john.fornass@ntnu.no}

\begin{abstract}
Transcendental H\'enon maps are the natural extensions of the well investigated complex polynomial H\'enon maps to the much larger class of holomorphic automorphisms. We prove here that transcendental H\'enon maps always have non-trivial dynamical behavior, namely that they always admit both periodic and escaping orbits, and that their Julia sets are non-empty and perfect.
\end{abstract}

\maketitle

\tableofcontents

\section{Introduction}
We investigate the dynamics of transcendental H\'enon maps $F\colon\C^2\ra\C^2$ defined as
\[
F(z,w)=(f(z)-\delta w, z),
\]
where $f\colon\C\ra\C$ is entire transcendental. The current paper continues the study started in \cite{henon1}, where the Julia set and Fatou set were introduced and various possible of Fatou components were constructed. Recall from \cite{FM89} that (compositions of) polynomial H\'enon maps are the polynomial automorphisms with non-trivial dynamical behavior. The goal of this paper is to show that transcendental H\'enon maps, a subclass of the holomorphic automorphisms, all have non-trivial behavior as well. We prove the following:

\begin{theorem}
The Julia set $J_F$ is non-empty and has no isolated points.
\end{theorem}

As a consequence the Julia set is uncountable. The result is a direct consequence of the following three propositions:

\begin{prop}\label{prop:one}
There exist periodic orbits.
\end{prop}

\begin{prop}\label{prop:two}
There exist escaping orbits.
\end{prop}

\begin{prop}\label{prop:three}
Every Fatou component is pseudoconvex.
\end{prop}

Note that the orbits inside a Fatou component containing a periodic orbit must all be bounded, hence Propositions \ref{prop:one} and \ref{prop:two} together imply that the Julia set is non-empty. The fact that the Julia set is perfect follows from Proposition \ref{prop:three}, since a punctured ball is not pseudoconvex.

We note that Propositions \ref{prop:one}, \ref{prop:two}, and \ref{prop:three} all hold for polynomial H\'enon maps, \ref{prop:two} following directly from the existence of the filtration \cite{HOV1}. In the polynomial setting pseudoconvexity of bounded Fatou components is again immediate, for unbounded components it follows from pluriharmonicity of the pluricomplex Green function $G^+$. Alternatively pseudoconvexity follows from the fact that the Julia set $J^+$ equals the support of a positive closed $(1,1)$-current. Observe that none of the proofs from the polynomial setting hold for transcendental maps, and indeed, the proofs given here are quite different.

Propositions \ref{prop:one} and \ref{prop:two} will be proved in section 2, relying heavily on Wiman-Valiron Theory for transcendental maps in one variable. We prove in fact that there must always be periodic points of order $1$, $2$ or $4$. The existence of escaping orbits closely follows ideas of \cite{Eremenko89}. In section 3 we prove that through every of the escaping points we construct, which we will refer to as Eremenko escaping points, there exists a complex curve of escaping points obtained as a strong stable manifold. We note that it remains open whether all points in the escaping set are connected to infinity. The same question is open for transcendental maps in one complex variable, where this is known as Eremenko's Conjecture.

In section 4 we prove that the Julia set can be equal to all of $\mathbb C^2$, by constructing a saddle fixed point with a dense stable manifold. Recall that the Julia set of an entire function in one complex variable is either the whole complex plane or has no interior; a fact that holds trivially for polynomials and for polynomial H\'enon maps similarly. Whether the same dichotomy holds for transcendental H\'enon maps remains to be decided.

In section 5 we prove Proposition \ref{prop:three}, which implies in particular that there are no isolated points in the Julia set. A main ingredient on the proof is the existence of an escaping orbit for the inverse of $F$, which is guaranteed by Proposition \ref{prop:two}. We emphasize that the Fatou components in Proposition \ref{prop:three} refer to the $\mathbb P^2$-Fatou set. We will discuss the distinction between the different definitions of the Fatou set in section 5. We note that the other statements in this paper hold regardless of which compactification of $\mathbb C^2$ is considered for normality, and hence we will not specify it.

\section{The Julia set is not empty}

The proofs of Propositions \ref{prop:one} and \ref{prop:two} rely on main results in Wiman-Valiron Theory in one complex variable, according to which entire transcendental  functions behave almost like polynomials of very high degree near points whose image has large modulus. We will start with a brief recollection of results in Wiman-Valiron Theory, followed by the proofs of Propositions \ref{prop:one} and \ref{prop:two}.

\subsection{Basics of Wiman-Valiron Theory}\label{sect:WV}
For the contents of this Section we refer to \cite{Eremenko89}, \cite{Langley07}.

Let $f(z)=\sum_{0}^\infty a_n z^n$ be an entire transcendental function.
For any $r>0$, the terms $|a_n |r^n\ra 0$ as $n\ra\infty$, hence for any $r$ there is a maximal term. Let $N(r)$ be the index of the maximal term. When such a maximal term is not unique, choose the largest index. The function $N(r)$ is called the \emph{central index} of $f.$ It is an increasing function and $N(r)\ra\infty$ for $r\ra\infty$.

 Let $M(r)$ denote the \emph{maximum modulus} of $f$; that is,
\[M(r):=\max_{|z|=r} |f(z)|.\]
Since $f$ is a transcendental entire function, we have
\begin{equation}\label{eqtn:WVlogMr}
\lim_{r\to\infty}\frac{\log M(r)}{\log r}=\infty,
\end{equation}
in particular for every $k>0$, $M(r)>r^k$ for large $r$.

Recall that a  set $E\subset \R$ has \emph{finite logarithmic measure} if $\int_{[1,\infty]\cap E}\frac{dt}{t}<\infty$.

By \cite[Lemma 2.2.8] {Langley07} we have the following upper bound: for any $\epsilon$,
\begin{equation}\label{eqtn:M vs N}
N(r)\leq (\log M(r))^{1+\epsilon}
\end{equation}
outside a set $E_\epsilon$ of finite logarithmic measure (which tends to infinity as $\epsilon\ra0$).

The main result in Wiman-Valiron Theory is the following (\cite{Eremenko89}, \cite[Theorem 2.2.20]{Langley07}).
\begin{thm}[Wiman-Valiron Estimates]\label{thm:WV}
Let $f$ be entire transcendental, $\frac{1}{2}<\alpha<1$, and let $q$ be a positive integer. For $r > 0 $ let $\zeta_r$ be a point of maximum modulus for $r$, that is, such that $|\zeta_r|=r$ and $|f(\zeta_r)|=M(r)$. If $z$ satisfies
\begin{equation}
|z-\zeta_r|< \frac{r}{(N(r))^\alpha} \label{Size of WV disks}
\end{equation}
then
\begin{align}
\label{eqtn:WVEstimatef}  f(z)&=\left(\frac{z}{\zeta_r}\right)^{N(r)}f(\zeta_r)(1+\epsilon_0),\\
\label{eqtn:WVEstimatefj} f^{(j)}(z)&= \frac{N(r)^j}{\zeta_r^j}f(z)(1+\epsilon_j),
\end{align}
for all $1\leq j\leq q$, where  $\epsilon_i$ are functions converging uniformly to $0$ in $z$ as $r\ra\infty$, for $r$ outside an exceptional set $E$ of finite logarithmic measure.
\end{thm}
The disk $\left\{|z-\zeta_r|< \frac{r}{(N(r))^\alpha}\right\}$ is called a \emph{Wiman-Valiron disk}.
Equation~(\ref{eqtn:WVEstimatef}) can be rewritten as

\begin{equation}\label{eqtn:WVEstimateflog}
\log f(z)-\log f(\zeta_r) ={N(r)}\left(\log{z}-\log \zeta_r\right)+\log(1+\epsilon_0).
\end{equation}

\begin{rem}\label{univlog}
It follows by (\ref{eqtn:WVEstimatefj}) for $j=1$ that
$$(\log f)'(z)=\frac{N(r)}{\zeta_r}(1+\epsilon_1).$$ If $r$ is large enough so that $|\epsilon_1|<1$, this implies that $\log f$ is univalent on the Wiman-Valiron disk.
\end{rem}

\subsection{Existence of periodic orbits}

In this subsection we prove Proposition \ref{prop:one}, stating that every transcendental H\'enon map admits a periodic orbit.


By \cite[Proposition 3.3]{henon1}, a transcendental H\'enon map $F$ admits periodic points of order $1$ or $2$ unless it is of the form
$$
F: (z,w) \mapsto (e^{g(z)} + w, z),
$$
where $g\colon \C\to \C$ is an entire holomorphic function. Hence
in order to prove Proposition \ref{prop:one} it is enough to prove the following proposition.
\begin{prop} Let $F(z,w)= (e^{g(z)} + w, z)$. Then $F$ has infinitely many periodic points of order $4$.
\end{prop}
 Recall that the inverse of a Transcendental H\'enon map
 $$F(z,w)= (f(z)-\delta w,z)$$
 is given by
$$
F^{-1}(z,w)= (w,\frac{ f(w)-z}{\delta}).$$
\begin{proof}
 We consider the case in which    $g$ is transcendental. If $g$ is polynomial, the proof is similar but  simpler.

  Since $F$ is a homeomorphism, a periodic point  $(z,w)$ of order $4$ has to satisfy the equation $F^2(z,w)=F^{-2}(z,w)$, which by the special form of the H\'enon map reduces to the system
$$
\begin{cases}
e^{g(w+e^{g(z)})} = - e^{g(w)} \\
e^{g(-e^{g(w)}+z)} = - e^{g(z)}.
\end{cases}
$$
In particular any solution to the system
\begin{equation}\label{principale}
\begin{cases}
g(w + e^{g(z)}) - g(w) = \pi i\\
g(z-e^{g(w)}) - g(z) = -\pi i.
\end{cases}
\end{equation}
gives a periodic point of order $4$.


We   first   look at a first order approximation of these two equations  in the term $e^{g(z)}$, and look for a solution  on the diagonal $\{z=w\}$. With these simplifications both of the equations above reduce to the equation

\begin{equation}
\label{eqtn:periodicapprox}
g^\prime(z)  e^{g(z)} = \pi i.
\end{equation}
We now look for a solution of (\ref{eqtn:periodicapprox}). Let $\alpha=2/3$, $q=2$. Let $\epsilon_j$ and $E$ be given by Theorem \ref{thm:WV} for the function $g$.  
  By (\ref{eqtn:M vs N}) there is a set  $E'$  of finite logarithmic measure such that for all $r\not\in E'$
\begin{equation}\label{langleyestimate}
N\leq (\log M)^2.
\end{equation}

Let $r>0$ be such that $r\notin E\cup E'$ and such that
$$|\log(1+\epsilon_0)|<1,\quad |\epsilon_1|<1.$$
Let $\zeta_r$ be a point of maximum modulus for $r$, let $N=N(r)$ and consider a domain  $D$ of the form
$$
D := \{z : re^{-\frac{2}{N}} < |z| < re^{\frac{2}{N}}\; , \; \; |\mathrm{Arg}(z)-\mathrm{Arg}(\zeta_r)|< 4\pi/N\}.
$$
Observe that $D$ is contained in the Wiman-Valiron disk given by (\ref{Size of WV disks}).
The function $$h(z):= N(\log z-\log \zeta_r)+\log g(\zeta_r)$$ maps $D$ univalently onto a rectangle centered at $\log g(\zeta_r)$ of width 4 and height $4\pi$. By Remark \ref{univlog}  the map $\log g$ is univalent on $D$ since $|\epsilon_1|<1$.
Note that (\ref{eqtn:WVEstimateflog}) gives
\[
\log g=h+\log(1+\epsilon_0).
\]
Since $|\log(1+\epsilon_0)|<1$, the image  $\log g(D)$ contains a rectangle  centered at $\log g(\zeta_r)$ of width $2$ and height $3\pi$.
It follows that $g(D)$ contains the annulus
$$
A_r= \{M/e < |z| < M e\}.
$$
By (\ref{langleyestimate})
$$\frac{\log\frac{r}{NM}}{M}\to 0,$$
as $r\to \infty$.
Hence  in the image $g(D)$ we can find a closed rectangle $R$ of width $6$ and height $4\pi$ centered at a point $u_0 +i v_0$ with
\begin{equation*}
u_0 = \log\frac{\pi r}{NM} ,
\end{equation*}
with an inverse branch $g^{-1}\colon R\to g^{-1}(R)$ defined in a neighborhood of $R$. We can also assume that a neighborhood of $R$ or radius $\pi/4$ is still contained in $g(D)$.

Let $\gamma\colon \mathbb{S}^1\to \C$ be a positively oriented parametrization of $\partial g^{-1}(R)$.
Let $\gamma_{\rm left},\gamma_{\rm right}$ be the portions of $\gamma$ which are mapped by $g$ to the  left and to the right vertical
 sides of  $\partial R$ respectively.
Notice that $z$ solves the equation $g^\prime(z)  e^{g(z)} = \pi i$ if  it solves the equation
\begin{equation}\label{logs}
\log g'(z)+g(z)=\log(\pi i )+2k\pi i =\log \pi+i\frac{\pi}{2}+2k\pi i
\end{equation} for some $k\in\Z$. So now we will estimate the function $\log g'(z)+g(z)$ on $g^{-1}(R)$.

 By Theorem~\ref{thm:WV},
\[|g'(z)|= \left|\frac{z}{\zeta_r}\right|^{N} \frac{NM}{r}  |(1+\epsilon_0)(1+\epsilon_1)|.\]
  Then for every $z\in\gamma_{\rm left} $ we have that
 $${\rm Re}[ g(z)+\log g'(z)]=\log \frac{\pi r}{NM}-3+\log |g'(z)|\leq \log \pi-1+\log |(1+\epsilon_0)(1+\epsilon_1)|\leq\log \pi-\frac{1}{2} ,$$
 where we used that $|z|< re^{\frac{2}{N}}$ and assumed $r$ big enough.
Similarly, for every $z\in \gamma_{\rm right}$ we get that
$${\rm Re} [g(z)+\log g'(z)]\geq\log \pi +\frac{1}{2}.$$

 We now estimate $\Im[g(z)+\log g'(z)]$ on $g^{-1}(R)$.
By (\ref{eqtn:WVEstimatefj}) we have that ${\rm Arg}\, g'$ is close to a constant on $g^{-1}(R)$, and thus that  ${\rm Im} \log g'$ is close to  constant, say $c$.
Hence, since $\Im g(z)$ ranges between $v_0-2\pi$ and $ v_0+2\pi,$  the image of $g^{-1}(R)$ via the map $g(z)+\log g'(z)$ contains a rectangle of width 1 and height  $3\pi$ centered at   the point $\log \pi+i(v_0+c)$.
It follows that    there exists $z_0\in g^{-1}(R)$
satisfying (\ref{logs}) for some $k\in\Z$ and hence satisfying (\ref{eqtn:periodicapprox}).

 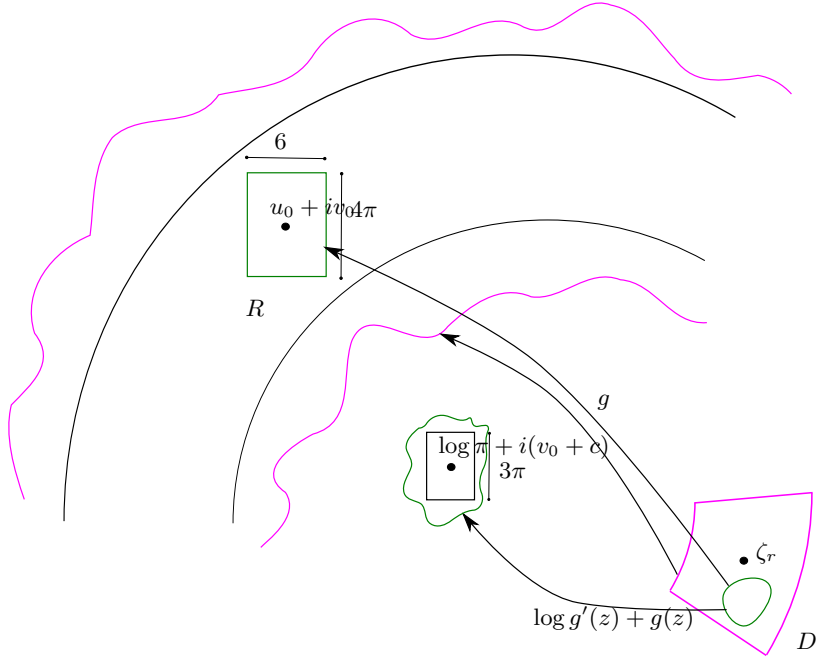
\begin{figure}[hbt!]
\begin{center}
\def\svgwidth{0.7\textwidth}
\begingroup%
  \makeatletter%
    \setlength{\unitlength}{\svgwidth}%
  \makeatother%
  \begin{picture} (1,0.77648714)%
    \put(0,0){\includegraphics[width=\unitlength]{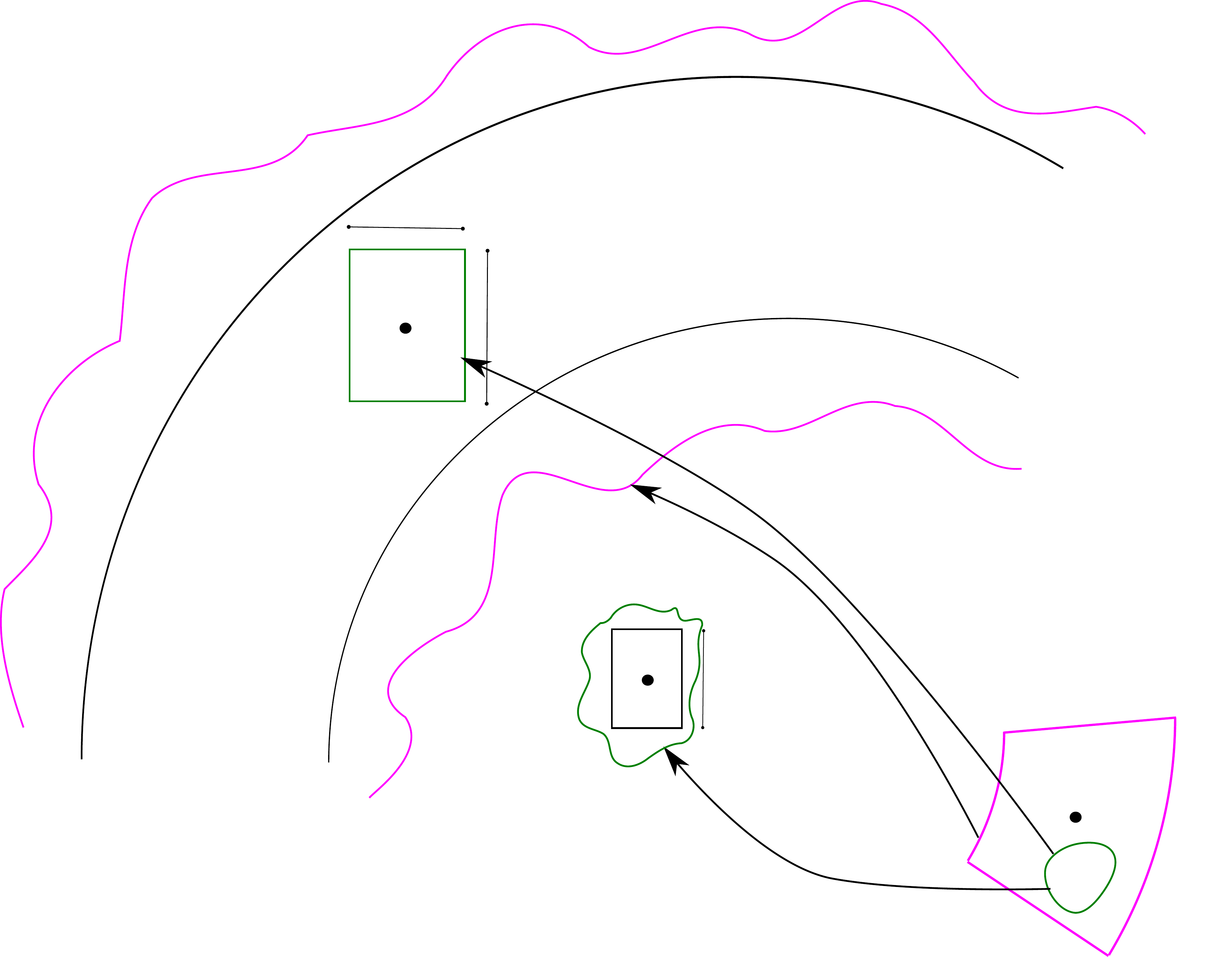}}%
    \put(0.93547219,0.0083928){\color[rgb]{0,0,0}\makebox(0,0)[lb]{\small{$D$}}}%
    \put(0.28177084,0.40398144){\color[rgb]{0,0,0}\makebox(0,0)[lb]{\small{$R$}}}%
    \put(0.88739487,0.11008297){\color[rgb]{0,0,0}\makebox(0,0)[lb]{\small{$\zeta_r$}}}%
    \put(0.31613106,0.60240328){\color[rgb]{0,0,0}\makebox(0,0)[lb]{\small{$6$}}}%
    \put(0.40655724,0.52202446){\color[rgb]{0,0,0}\makebox(0,0)[lb]{\small{$4\pi$}}}%
    \put(0.31038972,0.52202446){\color[rgb]{0,0,0}\makebox(0,0)[lb]{\small{$u_0+i v_0$}}}%
    \put(0.58291346,0.21098653){\color[rgb]{0,0,0}\makebox(0,0)[lb]{\small{$3\pi$}}}%
    \put(0.69959048,0.29267669){\color[rgb]{0,0,0}\makebox(0,0)[lb]{\small{$g$}}}%
    \put(0.62367779,0.0296613){\color[rgb]{0,0,0}\makebox(0,0)[lb]{\small{$\log g'(z)+g(z)$}}}%
    \put(0.51061912,0.23423955){\color[rgb]{0,0,0}\makebox(0,0)[lb]{\small{$\log\pi+i(v_0+c)$}}}%
  \end{picture}%
\endgroup%

\end{center}
\caption{\small Finding a solution $z_0$ to the first order approximation of equation~\ref{eqtn:periodicapprox}, not drawn to scale.}
\label{}
\end{figure}

%
%

Let $h_1, h_2$ be two real functions. We say that $h_1\simeq h_2$  as $x\to x_0$  if $h_1=O(h_2)$ and $h_2=O(h_1)$.
By our choice of the rectangle $R$ the disc $\D(g(z_0),\pi/4)$ of radius $\pi/4$ centered in $g(z_0)$ is also contained in $g(D)$ and admits an inverse branch of $g$.
If $z\in g^{-1}(\D(g(z_0),\pi/4))$ we have  that
\begin{equation}\label{orderestimate}
 |e^{g(z)}|\simeq \frac{r}{NM}.
\end{equation}
Since $(g^{-1})'(g(z_0))\simeq\frac{r}{NM}$, by the Koebe $\frac{1}{4}$-theorem $g^{-1}(\D(g(z_0),\pi/4))$ contains a disk $B$ centered at $z_0$ of radius $ \simeq \frac{r}{NM}$.

Our goal is now to find a solution to the original system (\ref{principale}) close to the point $(z_0, z_0)$.
Consider the map
\begin{equation}\label{eqtn:mapG}
G(z,w): = (g(w +e^{g(z)}) - g(w), g(z-e^{g(w)}) - g(z))
\end{equation}
  and let us show that there exists $(z,w)$ such that $G(z,w)=(\pi i , -\pi i)$.
Let $G_1$ and $G_2$ denote the components of $G$.

Using Taylor expansion with Lagrange remainder for $g$ near $z_0$ and the fact that $z_0$ satisfies (\ref{eqtn:periodicapprox}) gives
$$
|G_1(z_0,z_0)-\pi i|\leq \frac{\sqrt 2}{2}\max_{t\in [0,1]}|g''(z_0+te^{g(z_0)})||e^{2g(z_0)}|,
$$
and a similar estimate holds for $G_2$.  By Theorem~\ref{thm:WV} on $D$ we have
$$|g''(z)|=\frac{N^2}{r^2}\left(\frac{|z|}{r}\right)^NM|(1+\epsilon_0)(1+\epsilon_2)|\simeq \frac{N^2M}{r^2} .$$
 By (\ref{orderestimate}) it follows that
\begin{equation}\label{almostsolution}
\|G(z_0,z_0)-(\pi i,-\pi i)\|\simeq \frac{1}{M}.
\end{equation}

The differential $DG(z,w)$ equals
$$\left(
  \begin{array}{cc}
    g^\prime(w + e^{g(z)})e^{g(z)}g^\prime(z) & g^\prime(w + e^{g(z)}) - g^\prime(w) \\
    g^\prime(z - e^{g(w)}) - g^\prime(z) &  -g^\prime(z-e^{g(w)})e^{g(w)}g^\prime(w) \\
  \end{array}
\right)\simeq
\left(
  \begin{array}{cc}
 \frac{NM}{r} & \frac{N}{r} \\
\frac{N}{r}&  \frac{NM}{r} \\
  \end{array}
\right)
.
$$
We claim that for $r$ large enough the map $G$ is univalent on the polydisk $B\times B$.
Since for $z\in B$ the point $g(z)$ varies in a ball of radius $\pi/4$, it follows that $e^{g(z)}$ varies in a sector of angle $\frac{\pi}{2}$ of the complex plane. Hence for a large enough $r$,  the diagonal entries of the  differential $DG(z_1,z_2)$ are contained in two sectors  of angle $\frac{2\pi}{3}$ of the complex plane  for all $z= (z_1,z_2)\in B\times B$. Let $z\neq w$ be distinct points in $B\times B$, and assume first that $w_1-z_1 \neq 0$.
If $\gamma(t)=z+t(w-z)$ is the segment joining $z$ and $w$, then since the diagonal entries of $DG$ go to infinity faster than the off-diagonal entries we have that for all $t\in [0,1]$, the first component of $DG(\gamma(t))(w-z)$ is contained in an open sector of angle strictly less than $\pi$, which implies that $G(z)\neq G(w)$. We argue similarly if $z_1 = w_1$ but $z_2 \neq w_2$.

Since  for all $v\in \C^2$ there exists  $C(r)\simeq  \frac{NM}{r} $ such that $$\|DG(z,w)v\|\geq C\|v\|, \quad \forall (z,w)\in B\times B,$$
it immediately follows that  $G(B\times B)$ contains a ball of radius $\simeq 1 $ centered at $G(z_0,z_0)$.
It follows from (\ref{almostsolution}) that $(\pi i,-\pi i)$ is in the image of $B\times B$ when $r$ is big enough, which completes the proof.


\end{proof}
Notice that the freedom in the choice of $r_n$ and $D_0$ allows the construction of infinitely many periodic orbits.

\subsection{Existence of escaping points}

The proof of the following result, implying Proposition \ref{prop:two}, is inspired by Eremenko's proof  \cite{Eremenko89} of the fact that a transcendental function $f\colon \C\to \C$ admits an escaping point.

\begin{theorem}\label{eremenkotheorem}
Any transcendental H\'enon map admits   infinitely many escaping orbits, converging  to the point $[1:0:0]\in \ell^\infty$.
\end{theorem}

\begin{proof}
Let $F(z,w)=(f(z)-\delta w,z)$.
Let $\alpha=2/3$ and $q=1$. Let $\epsilon_0,\epsilon_1$ and $E$ be as in Theorem \ref{thm:WV} for the function $f$.
Finally, let $R>0$ be sufficiently large such that the following six properties are satisfied:
\begin{itemize}
\item[(i)]  $ |\log(1+\epsilon_0)|<\frac{1}{2}$ and $|\epsilon_0|\leq \frac{1}{2}$ for all $r\geq R$,
\item[(ii)] $ |\epsilon_1|<\frac{1}{2}$ for all $r\geq R$,
\item[(iii)]  for all $r\geq R$ we have that
$\frac{M(r)}{r}$ is larger than an arbitrarily large constant, to be determined later in the proof, 
\item[(iv)] the logarithmic measure of $E\cap [R,\infty)$ is $\leq 1$,
\item [(v)]  $N(R)>8$  (recall that $N(r)$ is increasing),
\item[(vi)] for all $r\geq R$  and for every point $\zeta_r$ with $|\zeta_r| = r$ the domain
$$D := \{z : re^{-\frac{2}{N(r)}} < |z| < re^{\frac{2}{N(r)}}\; , \; \; |\mathrm{Arg}(z)-\mathrm{Arg}(\zeta_r)|< 4\pi/N(r)\}$$
is contained in the disk $\{|z-\zeta_r|< \frac{r}{N(r)^{2/3}}\}$.
\end{itemize}

We construct inductively an increasing sequence of radii $r_n\geq R$, $r_n\not \in E$, $r_n\to \infty$,
 (denote $M_n$ and $N_n$ the maximum modulus and central index of $r_n$) with  points of maximum modulus $\zeta_n$,  a sequence of domains $(D_n)_{n\geq 0}$ defined as
  $$D_n := \{z : r_ne^{-\frac{2}{N_n}} < |z| < r_ne^{\frac{2}{N_n}}\; , \; \; |\mathrm{Arg}(z)-\mathrm{Arg}(\zeta_n)|< 4\pi/N_n\},$$
such that
 $$D_n\subset A_n:= \{M_{n-1}/e < |z| < M_{n-1} e\},\quad \forall n\geq 1,$$
and a sequence of univalent maps $(\varphi_n\colon D_{n}\to D_{n-1})_{n\geq 1}$ satisfying
 \begin{itemize}
 \item[(a)] $f(D_0)\supset A_1$ and  $(f-\delta \varphi_{n})(D_{n})\supset A_{n+1}$ for all $n\geq 1$,
\item[(b)] $\varphi_{1}\colon D_1\to D_0$ is an  inverse branch of $f$ and $\varphi_{n}$ is an  inverse branch of $f-\delta \varphi_{n-1}$ for all $n>1$,
\item[(c)] $|\varphi_n'(w)|\leq 1$ for all $ w\in D_n$.
\end{itemize}

Assuming that this is done, let us first show how these constructions imply the existence of an escaping point  for $F$. Consider for $n\geq 0 $ the embedded complex submanifold defined as the graph
$$
\Gamma_n:=\{(z, \varphi_n(z)),\, z\in  D_n\},
$$
so that for all $n\geq 1$ we have $F^{-1}(\Gamma_n)\subset \subset \Gamma_{n-1}$.
Indeed,
$$
F^{-1}(z,\varphi_n(z))=(\varphi_n(z),\varphi_{n-1}(\varphi_n(z)))\in \Gamma_{n-1}\cap \pi_z^{-1}(\varphi_n(D_n)).
$$

If  for all $n> 1$ we define the set
$$
K_n:={F^{-n}(\Gamma_n)}\subset \Gamma_0,
$$
then $K_n\subset \subset K_{n+1}$ for all $n$, and moreover the intersection $\bigcap_{n\in \N} K_n$ is nonempty and  consists of escaping points.

We now construct the radii $r_n$, the points $\zeta_n$, domains $D_n$ and the maps $\varphi_n$ recursively, starting with $r_0 \ge R$ outside of the exceptional set $E$.  Let  $\zeta_{0}$ be a point of maximum modulus for $r_0$ and consider the domain
$$
D_0 := \{z : r_0e^{-\frac{2}{N_0}} < |z| < r_0e^{\frac{2}{N_0}}\; , \; \; |\mathrm{Arg}(z)-\mathrm{Arg}(\zeta_0)|< 4\pi/N_0\}.
$$
By (vi) $D_0$ is contained in the Wiman-Valiron disk centered in $\zeta_0$.
As in the proof of Proposition \ref{prop:one} we obtain, using (i) and (ii) that $\log f$ is univalent on $D_0$ and that the image $\log f(D_0)$ contains a rectangle $Q$  centered at $\log f(\zeta_0)$ of width $2$ and height $3\pi$, and thus that $f(D_0)$ contains the annulus
$$
A_1:= \{M_0/e < |z| < M_0 e\}.
$$

The set $ \{M_0/e < r < M_0 e\}$ has logarithmic measure 2 and by (iii) it is contained in $  [R,\infty)$.   By (iv) $E\cap [R,\infty)$ does not contain any interval of logarithmic measure 1. By (v)  there exists $r_1\notin E$ such that
$$
D_1 := \{z : r_1e^{-\frac{2}{N_1}} < |z| < r_1e^{\frac{2}{N_1}}\; , \; \; |\mathrm{Arg}(z)-\mathrm{Arg}(\zeta_{1})|< 4\pi/N_1\}
$$
is compactly contained in the annulus $A_1$, where  $\zeta_{1}$ is a point of maximum modulus for $r_1$.
Notice that the image $f(D_0)$ winds around $A_1$ at least 1.5 times.  Since $4\pi/N_1<\pi/2$ by (v), it follows that there exists
 an inverse branch $\varphi_1\colon  D_1\to D_0$ of $f$.   Notice that
$|\varphi_{1}'|\leq 1$. Indeed  if  $z\in D_0$, then by (i) and (ii)    $|f'(z)|\geq \frac{e^{-2}N_0M_0}{4r_0}$, hence by (iii) $|f'(z)|$ can be assumed to be larger than $1$.

Suppose now that we have carried on the construction up to $n\geq 0$.
Observe that
$$
(\log(f-\delta\varphi_{n}))'=\left(\frac{f'}{f}-\frac{\delta\varphi'_{n}}{f}\right)\frac{f}{f -\delta\varphi_{n}}= \left(\frac{N_n}{\zeta_n}(1+\epsilon_1)-\frac{\delta\varphi'_{n}}{f}\right)\frac{f}{f -\delta\varphi_{n}}.
$$
Since  by (\ref{eqtn:WVEstimatef})
%
 we have  $\left|   \frac{\delta\varphi'_{n}}{f} \right|\leq    \frac{|\delta|}{2^{-1}e^{-2}M_n}$ and therefore
$$
\frac{N_n}{r_n}/  \left|   \frac{\delta\varphi'_{n}}{f} \right|  \geq \frac{N_nM_n2^{-1}e^{-2}}{r_n|\delta|}.
$$
Combining this equation with (ii) and (iii) plus the observation $\left|\frac{\varphi_n}{f}\right|\leq \frac{r_n}{eM_n}$ gives that
$(\log(f-\delta\varphi_{n}))'$ takes values in a sector of angle strictly less than $\pi$, and thus that $\log(f-\delta\varphi_{n})$ is univalent on $D_n$.

We now show that the image $\log(f-\delta\varphi_{n})(D_n)$ contains a rectangle $Q$  centered at $\log f(\zeta_n)$ of width $2$ and height $3\pi$, and thus that the image $(f-\delta\varphi_{n})(D_n)$ contains the annulus
$$
A_{n+1}= \{M_n/e < |z| < M_n e\}.
$$
The function
$$
h(z):= N_n(\log z-\log \zeta_n)+\log f(\zeta_n)
$$
maps $D_n$ univalently onto a rectangle centered at $\log f(\zeta_n)$ of width 4 and height $4\pi$. Since by (\ref{eqtn:WVEstimateflog}) we have that $\log f=h+\log(1+\epsilon_0)$ it follows that
$$
\log(f-\delta\varphi_n)=h +\log(1+\epsilon_0)+\log (1- \frac{\delta\varphi_n}{f}),
$$
and by (i) and (iii) we have that  $|\log(1+\epsilon_0)+\log (1- \frac{\delta\varphi_n}{f})|< 1$.

The set $ \{M_1/e < r < M_1 e\}$ has logarithmic measure 2 and by (iii) it is contained in $ [R,\infty)$. By  (iv) the set $E\cap [R,\infty)$ does not contain any interval of logarithmic measure 1. By (v) there exists $r_{n+1}\notin E$ such that  the domain
$$
D_{n+1} := \{z : r_{n+1}e^{-\frac{2}{N_{n+1}}} < |z| < r_{n+1}e^{\frac{2}{N_{n+1}}}\; , \; \; |\mathrm{Arg}(z)-\mathrm{Arg}(\zeta_{n+1})|< 4\pi/N_{n+1}\}
$$
is compactly contained in the annulus $A_{n+1}$, where  $\zeta_{n+1}$ is a point of maximum modulus for $r_{n+1}$.
Let $\varphi_{n+1}\colon  D_{n+1}\to D_n$ be an inverse branch of $f-\delta \varphi_{n}$.
We only need to show that $|\varphi_{n+1}'|\leq 1$. Indeed, if $z\in D_{n+1}$ and $w:=\varphi_{n+1}(z)$,
\begin{equation}\label{eqtn:original estimates about phi}
|\varphi_{n+1}'(z)|\leq \frac{1}{|f'(w)|-|\delta\varphi'_{n}(w)|},
\end{equation}
and the claim follows since by (i), (ii) and (iii),   $|f'(w)|\geq \frac{e^{-2}N_nM_n}{4r_n}\geq 1.$

Notice finally that each of the infinitely many choices of the $D_n$ gives different orbits of escaping points.
 \end{proof}

For every subset $X\subset \C$, we denote by $\mathcal{N}_\delta(X)$ its $\delta$-neighborhood.
In the next section we will need better estimates of the previous construction.
It is easy to see that we can arrange the construction in such a way  that $\mathcal{N}_\delta(D_{n+1})\subset A_{n+1}$, and thus the inverse branch $\varphi_{n+1}$ of $f-\delta\varphi_n$ is defined on the neighborhood $\mathcal{N}_\delta(D_{n+1})$.
Moreover, up to  taking $n$ large enough, there exists a constant $C>1$, independent of $n$, such that
\begin{equation}\label{jonsnow}
 C^{-1}\frac{M_nN_n}{r_n}  \leq   |f'(z)-\delta\varphi_{n}'(z)|\leq C\frac{M_nN_n}{r_n},\quad \forall z\in D_n.
\end{equation}
and thus
\begin{equation}\label{sam}
C^{-1}\frac{r_n}{M_nN_n}  \leq   |\varphi'_{n+1}(w)|\leq C\frac{r_n}{M_nN_n},\quad \forall w\in \mathcal{N}_\delta(D_{n+1}).
\end{equation}

\begin{defn}
We refer to any escaping point constructed in this way as
an {\sl Eremenko escaping point}.
\end{defn}

\section{Curves of escaping points}

For complex (polynomial) H\'enon maps, the escaping set $U^+$ is foliated by embedded complex lines, a result of Hubbard and Oberste-Vorth \cite{HOV1}. One usually constructs these complex lines by considering the unique complex tangent directions of level curves $\{G^+ = c\}$, where $G^+$ is the forward Green's function defined by
$$
G^+(z) = \lim_{n \rightarrow \infty} \frac{\log^+\|F^n(z)\|}{\mathrm{deg}(F)^n}  .
$$
The fact that $G^+ : U^+ \rightarrow \mathbb R$ is a pluriharmonic submersion implies the existence of the leaves of the foliation.

A different point of view, with more potential for generalization to the transcendental setting, is to consider the leaves of the foliation as \emph{strong stable manifolds}. Let $(P_n)$ be a forward orbit contained in $U^+$. Up to renumbering the sequence $(P_j)$ we may assume that $|z_0|$ is large, and that $|z_0| >> |w_0|$. It follows that the polynomial $f(z)$ is strongly expanding near each $z_j$, and we may assume that image under $f$ of the disk $D(z_j, 1)$ covers $D(z_{j+1},1)$ univalently for each $j \ge 0$.

For the two-dimensional map $F(z,w) = (f(z) - \delta w, z)$ it follows that images $F( \Delta^2(P_j, 1))$ intersect the next bidisk $\Delta^2(P_{j+1},1)$, with underflow in the vertical direction and overflow in the horizontal, both with uniform estimates. Thus, by taking the inverse images of the straight vertical disks $\{(z_n, w_n+t) \mid |t| < 1\}$ and intersecting with the bidisks $\Delta^2(P_j,1)$, we obtain a family of properly embedded vertical graphs. For fixed $j$ these graphs form a Cauchy sequence, and the limits give the local strong stable manifolds through the points $P_j$.

We will try to mimic this construction for transcendental H\'enon maps, as long as we have an escaping orbit $(P_j)$ with strong horizontal expansion. Notice that it is not sufficient to control only the derivatives at the point $P_j$. We will need neighborhoods similar to the bidisks $\Delta^2(P_j,1)$ as above, with overflow in the horizontal and underflow in the vertical direction. It turns out that for the escaping points constructed in Theorem~\ref{eremenkotheorem} we have exactly the right information needed to prove the existence of the stable manifolds. Our proof, as the proof sketched above, closely follows the graph transform method. Most references in the literature deal with the setting where the orbit $P_j$ is either periodic or remains in a compact subset. For the reader's convenience we  give a detailed proof that does not rely on any sources, without claiming that the methods presented here are new. The contents of this section rely on the construction in the proof of Theorem~\ref{eremenkotheorem}.

\begin{theorem}\label{thm:Curves of escaping points}
Let $F(z,w)=(f(z)-\delta w,z)$ be a transcendental H\'enon map. Let $P_0$ be an Eremenko escaping point. Then there exists an injective holomorphic immersion $\gamma\colon \C\to \C^2$ whose image $\Sigma^s(P_0)$ contains $P_0$ and is contained in the escaping set $I(F)$. For every $Q\in \Sigma^s(P_0)$, $\|F^n(Q)-F^n(P_0)\|$ converges to zero at least exponentially fast as $n\ra\infty$.
\end{theorem}

\begin{proof}[Proof of Theorem~\ref{thm:Curves of escaping points}]
Let $P_0$ be an Eremenko escaping point and  for all $n\geq 0$ let $\Gamma_n, D_n,\varphi_n$ be defined as in the proof of Theorem \ref{eremenkotheorem}. In particular,  $P_0=\bigcap_{n\geq 0} F^{-n}(\Gamma_n)\subset \Gamma_0$.

Since $F$ is an automorphism it is sufficient to prove the claim up to replacing $P_0$ with a  point $P_n:=F^n(P_0)$ for some $n$ large enough.

We define the domains
$$
\begin{aligned}
U_n & := \{(z,w +t) : (z,w) \in \Gamma_n, |t| < 1\}\\
& = \{ (z,  \varphi_n(z)+t) : z \in D_n, |t| < 1\}.
\end{aligned}
$$
Notice that the shear $\Phi_n (z,t) := (z, t+\varphi_n(z))$  is  a biholomorphism from    $\U_n:=D_n \times \mathbb D$ to $U_n$ and let us define
$$ \tilde{P}_n:=(\Phi_n)^{-1}(P_n).
$$

Suppose that the point $(z, t+\varphi_n(z))$ belongs to $ U_n \cap F^{-1}(U_{n+1})$. Then clearly $F\circ\Phi_n(z,t)\in U_{n+1}$, hence $\tilde{z}:=f(z)-\delta\phi_n(z)-\delta t\in D_{n+1}$. Since  $\phi_{n+1}$ is defined in a $\delta$-neighborhood of $D_{n+1}$  we have that $\phi_{n+1}(\tilde{z}+\delta t)$ is well defined for $|t|<1$, and hence we can write
$$
F\circ \Phi_n(z,t)=F(z, t+\varphi_n(z))  = (f(z) - \delta \phi_n(z)- \delta t, z)
= (\tilde{z}, \phi_{n+1}(\tilde{z}) + \tilde{t}),
$$
where $\tilde{t}:=z-\varphi_{n+1}(\tilde z).$
This induces a map $\tilde{F}: \Phi_n^{-1}(U_n \cap F^{-1}(U_{n+1}))\ra \U_{n+1}$ defined as
$$
\tilde F := \Phi_{n+1}^{-1} \circ F \circ \Phi_n :(z,t) \mapsto (\tilde{z},\tilde{t}).
$$

\begin{lemma}\label{continuitylemma}
Let $(z,t+\varphi_{n}(z))$ be a point in the connected component  containing $P_n$ of $ U_n \cap F^{-1}(U_{n+1})$.
Then $$\varphi_{n+1}(f(z)-\delta\varphi_n(z))=z.$$
\end{lemma}
\begin{proof}
Let $\gamma\colon [0,1]\to U_n$ be a continuous curve joining the points  $P_n$ and $(z,t+\varphi_{n}(z))$, all contained inside the connected component. Denote $$\gamma(s):=(z(s), t(s)+\varphi_n(z(s)).$$ For all
$s\in [0,1]$, the point $f(z(s))-\delta\varphi_n(z(s))$ is contained in $\mathcal{N}_\delta(D_{n+1})$. Consider the equation
$$\varphi_{n+1}(f(z(s))-\delta\varphi_n(z(s)))=z(s).$$
This holds for $s=0$ since $z(0)=P_n$ and by continuity it holds for $s=1$.
\end{proof}
\begin{lemma}
Let $(z_1, t+\varphi_n(z_1))$ and $(z_2, t+\varphi_n(z_2))$ both belong to the connected component  containing $P_n$ of $ U_n \cap F^{-1}(U_{n+1})$.
There exist a constant $\Theta=\Theta(n) \stackrel{n\to\infty}\longrightarrow +\infty$    and a constant  $\Theta'>0$ independent of $n$,
such that the following estimates hold:
\begin{align}
\label{eqtn:horizontal overflow} |\tilde{z}_1-\tilde{z}_2|&\geq \Theta |z_1-z_2|;\\
\label{eqtn:vertical underflow}|\tilde{t_1}-\tilde{t_2}|&\leq  \Theta' |z_1-z_2|.
\end{align}
\end{lemma}

\begin{proof}
By definition of $\tilde{z}_1,\tilde{z_2}$ we have
\begin{equation}\label{daenerys}
|\tilde{z}_1-\tilde{z}_2|=|f(z_1)-\delta\phi_n(z_1)-(f(z_2)-\delta\phi_n{z_2})|.
\end{equation}
Since  $(z_1, t+\varphi_n(z_1))$ and $(z_2, t+\varphi_n(z_2))$  belong to the connected component  containing $P_n$ of $ U_n \cap F^{-1}(U_{n+1})$, by Lemma \ref{continuitylemma} we have that
 $z_1=\varphi_{n+1}(\tilde z_1+\delta t)$ and that  $z_2=\varphi_{n+1}(\tilde z_1+\delta t)$.
Let $\gamma\colon[0,1]\to \C$ be the segment joining the points $\tilde z_1+\delta t$ and $\tilde z_2+\delta t$. Then
$$|z_2-z_1|\leq \int_0^1 |\varphi'_{n+1}(\gamma(t))|dt\leq|\tilde z_2-\tilde z_1| \max_{\mathcal{N}_\delta(D_{n+1})} |\varphi'_{n+1}|.$$
Thanks to (\ref{sam}), setting $\Theta:= \frac{C^{-1}M_nN_n}{r_n}$ we obtain (\ref{eqtn:horizontal overflow}).
We now prove (\ref{eqtn:vertical underflow}).
We have $$|\tilde t_1-\tilde t_2|=|z_1-\varphi_{n+1}(\tilde z_1)-(z_2-\varphi_{n+1}(\tilde z_2))|\leq |z_1-z_2|+|\varphi_{n+1}(\tilde z_2)-\varphi_{n+1}(\tilde z_1)|.$$ Hence we are reduced to estimate $|\varphi_{n+1}(\tilde z_2)-\varphi_{n+1}(\tilde z_1)|$.
Arguing as above we obtain $$|\varphi_{n+1}(\tilde z_2)-\varphi_{n+1}(\tilde z_1)|\leq|\tilde z_2-\tilde z_1| \max_{\mathcal{N}_\delta(D_{n+1})} |\varphi'_{n+1}|.$$
From (\ref{daenerys}) it follows that
$$|\varphi_{n+1}(\tilde z_2)-\varphi_{n+1}(\tilde z_1)|\leq\max_{D_n}|f'-\delta\varphi'_n| \max_{\mathcal{N}_\delta(D_{n+1})} |\varphi'_{n+1}||z_2-z_1|.$$
The result therefore follows from (\ref{jonsnow}) and (\ref{sam}) setting $\Theta':=C^2+1$.
\end{proof}
\begin{lemma}\label{lemmaproper}
Let $Q:=(z, t+\varphi_n(z))$ belong to the connected component  containing $P_n$ of $ U_n \cap F^{-1}(U_{n+1})$.
Then there is a contact $\alpha=\alpha(n) \stackrel{n\to \infty}\longrightarrow 0$ such that
$$|\tilde t|<\alpha |\delta t|.$$
\end{lemma}
 \begin{proof}
 Consider the point $F(Q)=   (\tilde z,z)\in U_{n+1}$.  Since $Q$ belongs to the connected component  containing $P_n$ of $ U_n \cap F^{-1}(U_{n+1})$, by Lemma \ref{continuitylemma} we have that $z=\varphi_{n+1}(\tilde z+\delta).$
Hence
 $$|\tilde t|= |z-\varphi_{n+1}(\tilde z)|=|\varphi_{n+1}(\tilde z+\delta t)-\varphi_{n+1}(\tilde z)|\leq \max_{\mathcal{N}_\delta(D_{n+1})}|\varphi'_{n+1}||\delta t|.$$

 \end{proof}

Let $L_n = L_{n,n}$ be the intersection with $\U_n$ of the straight vertical line through $\tilde{P}_n$, and recursively define $L_{n,j}$ for all $n\geq j\geq 0$ as the connected component of $\tilde{F}^{-1} ( L_{n,{j+1}}) \cap \U_j$ that contains the point $\tilde{P}_j$. We will show that for $j_0$ sufficiently large the pullbacks $L_{n,j_0}$ converge to a proper holomorphic disk in $\mathbb U_{j_0}$, corresponding to the local stable manifold through $P_{j_0}$.

We claim that for each $n,j$ the curve $L_{n,j}$ is vertical, that is, its tangent space is contained in the vertical cone field
$$
C_{ver}(z,t) = \{(v_1, v_2)\in\C^2: |v_1| \ge|v_2| \},
$$
 and that $L_{n,j}$ it is a graph of the $t$-direction.
Since each  $L_{n,j}$ is a pullback of the vertical curve $L_{n,n}$,  it is enough to show  that  the vertical cone field is  backward invariant.

Observe that
\begin{displaymath}
D\tilde{F}(z,t)=
\left( \begin{array}{cc}
f'(z)  - \delta \varphi_n^\prime(z) & -\delta   \\
1 -\varphi^\prime_{n+1}(\tilde{z}) (f'(z)-\delta\phi_n'(z))  &   \delta \varphi^\prime_{n+1}(\tilde{z})
 \end{array} \right)=:\left( \begin{array}{cc}
A & a_1   \\
a_2  &  a_3
 \end{array} \right).
\end{displaymath}
By equations \eqref{jonsnow} and \eqref{sam} it follows that $\frac{a_i}{A}\ra0$ as $n\ra\infty$.

Denote by $\pi_2(z,w):=w$ the projection on the second variable. The fact that the $\delta$-neighborhood $\mathcal{N}_\delta(D_{n+1})$ satisfies $\varphi_{n+1}(\mathcal{N}_\delta(D_{n+1})) \subset D_n$ and $(\varphi_{n+1}(\partial \mathcal{N}_\delta(D_{n+1}))\times \mathbb D ) \cap U_{n+1} = \emptyset$ implies overflow in the horizontal direction. Combined with the underflow in the vertical direction given by Lemma \ref{lemmaproper}, this implies that the restriction $\pi_2\colon L_{n,j}\to \D$ is a proper holomorphic map, and thus a branched covering.
Since the curve $L_{n,j}$ is almost vertical, the covering cannot have branch points. Hence $\pi_2\colon L_{n,j}\to \D$ is a biholomorphism, and thus $L_{n,j}$ is a graph over $\D$.

We introduce the graph distance
$$
\mathrm{dist}(L_{n,j}, L_{m,j}) := \sup_{|t| < 1} \{|z_n - z_m| \mid (z_n,t) \in L_{n,j}, (z_m, t) \in L_{m,j}\}.
$$
As the tangent spaces to each of the graphs $L_{n,j}$ are vertical and they all pass through the point $P_j$, it follows that $\mathrm{dist}(L_{n,j}, L_{m,j})$ is bounded from above by the constant $1$, independently of $n,m,j$.

\begin{lemma}
We have that
\begin{equation}\label{eqtn:Cauchy first step}
\mathrm{dist}(L_{n,j}, L_{m,j}) \le \frac{1}{2}\cdot \mathrm{dist}(L_{n,j+1}, L_{m,j+1}).
\end{equation}
\end{lemma}
\begin{proof}
Let $(z_1, t)$ be the coordinates   of a point  $Q_1\in L_{n,j}$ and $(z_2,t)$ be the coordinates  of a point $Q_2 \in L_{m,j}$. Denote their images under $\tilde{F}$ by $(\tilde{z}_1, \tilde{t}_1) \in L_{n,j+1}$ and $(\tilde{z}_2, \tilde{t}_2) \in L_{m,j+1}$, and let $\tilde{z}_3$ be the unique value for which $(\tilde{z}_3, \tilde{t}_1) \in L_{m,j+1}$. The claim follows if we can show that  $|\tilde{z}_1-\tilde{z}_3|\geq 2|{z}_1-{z}_3|$.

 By  (\ref{eqtn:horizontal overflow})   it follows that
$$
|\tilde{z}_1 - \tilde{z}_2| \ge \Theta \cdot |z_1 - z_2|,
$$
while by (\ref{eqtn:vertical underflow}) it follows that

$$
|\tilde{t}_1 - \tilde{t}_2| \le \Theta' \cdot |z_1 - z_2|.
$$

Since the tangent space to $L_{m,j+1}$ is contained in the  vertical cone field and the two points $(\tilde{z}_2,\tilde{t}_2)$ and $(\tilde{z}_3, \tilde{t}_1)$ belong to $L_{m,j+1}$, it follows that
$$
|\tilde{z}_2 - \tilde{z}_3| \le |\tilde{t}_2-\tilde{t}_1|.
$$
It follows that
$$
|\tilde{z}_1 - \tilde{z}_3| \ge |\tilde{z}_1 - \tilde{z}_2| - |\tilde{z}_2 - \tilde{z}_3| \ge
\Theta|z_1-z_2|- |\tilde{t}_2-\tilde{t}_1| \ge  (\Theta-\Theta')|z_1-z_2|.
$$
Since $\Theta \ra \infty$ as $n \rightarrow \infty$ while $\Theta^\prime$ is constant, we may assume that $\Theta \ge 2+ \Theta'$, proving (\ref{eqtn:Cauchy first step}).
\end{proof}

As a consequence for fixed $j$ the graphs $L_{n,j}$ form a Cauchy sequence, and converge to a limit graph $\Sigma_j$. We denote by $\Sigma_{loc}^s(P_j)$ the graph $\Phi_j(\Sigma_j)$.

We now conclude the proof by showing that $\Sigma_{loc}^s(P_j)$ can be extended to a full complex curve $\Sigma^s(P_j)$ of escaping points, and that for any $Q\in\Sigma^s(P_j)$ we have that $\|F^n(P)-F^n(Q)\|\ra0$ at least exponentially fast as $n\ra\infty$.

By construction we have that $\tilde{F}(\Sigma_j) \subset \Sigma_{j+1}$, and  by (\ref{lemmaproper}) we obtain
$$
\tilde{F}(\Sigma_j) \subset \Sigma_{j+1} \cap \left\{(z,t):|t|< \frac{1}{2}\right\},
$$
which  by definition of $\tilde{F}$ implies that
\begin{equation} \label{eqn:kobayashi}
F(\Sigma_{loc}^s(P_j)) \subset \Sigma_{loc}^s(P_{j+1})\cap \left\{(z,t+\phi_n(z)):|t|< \frac{1}{2}\right\}.
\end{equation}

It follows that the global stable manifold through $P_j$ defined as
$$
\Sigma^s(P_j) := \bigcup_{n \ge j} F^{j-n} \Sigma_{loc}^s(P_n)
$$
is an increasing union of disks, and equation \ref{eqn:kobayashi} implies that the Kobayashi metric vanishes identically. Thus $\Sigma^s(P_j)$ is biholomorphic to $\mathbb C$. In particular the set $\Sigma^s(P_j)$ is unbounded and consists of escaping points.

Finally, we need to show the existence of  $C>0,\lambda<1$ such that for $Q\in \Sigma^s(P_j) $ we have
$$\|F^n(Q)-F^n(P_j)\|\leq C\lambda^n.$$

By the definition of the shear $\tilde{F}$ it is enough to check that for $Q\in \Sigma_j $ we have
$$
\|\tilde{F}^n(Q)-\tilde{F}^n(\tilde{P}_j)\|\leq C\lambda^n.
$$
To prove the latter let $\tilde{F}^n(Q)=(z'_n, t'_n)$ and $\tilde{F}^n(\tilde{P}_j)=(z_n,0)$. By (\ref{eqn:kobayashi}) we have $|t_n'|\leq\frac{1}{2^n}t_0'$, and the fact that the $\Sigma_j$ are vertical graphs passing through $\tilde{P}_j$ implies $|z_n'-z_n|\leq|t_n'|$. This concludes the proof.

\end{proof}

\section{Dense Stable Manifold}

We denote by $\mathcal{H}(\mathbb C^2)$ the space of all entire H\'enon maps equipped with the compact open topology. The subspace consisting of those H\'enon maps with Jacobian determinant $\delta$ will be denoted by $\mathcal{H}_\delta(\mathbb C^2)$. Throughout this section we will consider a fixed $\delta$ with $|\delta| > 1$.

\begin{theorem}\label{thm:4.1}
There exists a  dense $G_\delta$-subset  $\mathcal{U}$ of $\mathcal{H}_\delta(\mathbb C^2)$ such that
every map $F\in \mathcal{U}$ has a  fixed saddle point  whose stable manifold  is dense in $\C^2$.
\end{theorem}

In particular the Julia set any map $F\in \mathcal{U}$ is equal to $\C^2$.

\medskip

The analogous statement was proved for volume preserving holomorphic automorphisms of $\mathbb C^n$ in \cite{PVW}, and generalized to Stein manifolds with the volume density property in \cite{AL}. Our proof closely follows that in \cite{AL}, with only mild variations due to the fact that the maps under consideration are of a more restrictive form. In these steps the fact that our maps are volume expanding simplifies the proof considerably. We will give the proof, emphasizing the differences with the proof in \cite{AL}.

\medskip

\noindent {\bf Step 1.} We note that $\mathcal{H}_\delta(\mathbb C^2)$ is  homeomorphic to the Fr\'echet space of entire functions on $\mathbb C$. As such it a Baire space, it is  separable, and its topology is induced by a complete distance $d$.
 Let $\{\varphi_j\}_{j \in \mathbb N}$ be a dense countable subset.

\medskip

\noindent {\bf Step 2.} Using Runge approximation we can perturb each $\varphi_j$ to obtain an entire H\'enon map $\psi_j$ with $d(\psi_j, \varphi_j) < \frac{1}{j}$, such that each $\psi_j$ has a saddle fixed point. Note that the saddle fixed point may be constructed arbitrarily close to infinity. The collection $\{\psi_j\}_{j \in \mathbb N}$ is still dense in $\mathcal{H}_\delta(\mathbb C^2)$.

\medskip

\noindent {\bf Step 3.} For each $\psi_j$ we can find an arbitrarily small ball $B_j\subset \mathcal{H}_\delta(\mathbb C^2)$ in which the chosen saddle point $\psi_j$ can be followed continuously, and the corresponding local stable manifold is a graph over a fixed direction in a linearly embedded polydisk $\Delta^2_j$ which does not depend on the map $F \in B_j$. For $F \in B_j$ we denote the marked saddle fixed point by $\eta(F)$ and its stable manifold by $\Sigma^s_F(\eta(F))$.

\medskip

\noindent {\bf Step 4.} By recursively decreasing the radius of $B_j$ if $\psi_j$ is not contained in the closure of an earlier defined ball, or dropping $\psi_j$ if it is, we obtain a countable collection of pairwise disjoint balls $B_j$, whose union $A = \bigcup B_j$ is dense in $\mathcal{H}_\delta(\mathbb C^2)$. Note that since the balls $B_j$ are pairwise disjoint, the notation introduced in the previous step is now justified.

\medskip

Since  $\mathcal{H}_\delta(\mathbb C^2)$ is a Baire space, Theorem \ref{thm:4.1} is now is a direct consequence of the following lemma.

\begin{lemma}
Let $q \in \mathbb C^2$ and $\epsilon >0$. The subset $U(q,\epsilon) \subset \mathcal{H}_\delta(\mathbb C^2)$ given by
$$
U(q,\epsilon) = \{f \in A : {\rm dist}(\Sigma_F(\eta(f)),q) < \epsilon\}
$$
is open and dense.
\end{lemma}
\begin{proof}
The fact that $U(q,\epsilon)$ is open follows from the fact that stable manifolds vary continuously under perturbations of the map. To prove density, consider any $F \in B_j$, fix a compact subset $K \subset \mathbb C^2$ and a $\delta>0$. We will prove that there exists $\tilde{F} \in B_j$ with $\|\tilde{F} - F\|_K < \delta$ for which there exists $\tilde{q} \in \Sigma_{\tilde{F}}(\eta(\tilde{F}))$
such that $\|\tilde{q} - q\| < \epsilon$.

Let $R>0$ be such that $K$ and the polydisk containing the local stable manifold of $\eta(F)$ are both contained in the cylinder $C_R := \{|z| < R\}$. Since $|\delta| > 1$ there is a point $\tilde{q}$ arbitrarily close to $q$ whose forward orbit leaves
the polydisk $\{|z|<R,|w|<R\}$. Since $F$ is a H\'enon map, it follows that the forward orbit of $q$ also leaves the  cylinder $C_R$,  say $F^N(\tilde{q}) = (z_0, w_0)$, with $|z_0|>R$.

We remark that this step in the proof is much simpler in our setting than it is for volume preserving automorphisms, where ideas from \cite{ClosingLemma} were used to prove that for an arbitrary small perturbation of $F$ the orbit of $q$ leaves an arbitrarily large ball. We also note that the stable manifold $\Sigma_F(\eta(F))$ must contain a point $(z_1,w_1)$ with $|w_1| > R$. If this was not the case, the stable manifold would have to be a straight horizontal line, which is impossible since it is invariant under an entire H\'enon map.

Let $N>0$ be such that the point   $(z_N, w_N):= F^{N-1}(z_1,w_1)$ lies in the local stable manifold.
By wiggling the point $(z_1, w_1)$ on $\Sigma_F(\eta(F))$ we may assume that for all points in the forward orbit
$$
(z_1, w_1) ,\ldots, (z_N, w_N),
$$
their $z$-coordinates are pairwise disjoint and disjoint from the $z$-coordinates of the points in the finite orbit from $\tilde{q}$ to $(z_0, w_0)$.

By Runge Approximation we can therefore find an entire function $\tilde{f}$ satisfying the following properties:
\begin{enumerate}
\item[(i)] $\tilde{f}$ is arbitrarily close to $f$ on the disk $\{|z| < R\}$.
\item[(ii)] $\tilde{f}(z_0) = w_1 + \delta w_0$ and $\tilde{f}(w_1) = z_1 + \delta z_0$.
\item[(iii)] $\tilde{f}$ is arbitrarily close to a constant on small discs centered at $z_0$ and $w_1$.
\item[(iv)] $\tilde{f}$ is equal to $f$ on the finite orbit from $(z_1, w_1)$ to $(z_N, w_N)$ and the finite orbit from $\tilde{q}$ to $(z_0, w_0)$, and is arbitrarily close to $f$ on given neighborhoods of these points.
\end{enumerate}

Condition (i) implies that $\tilde{F}$ can be chosen in $B_j$.

Consider $\mathcal{N}(\tilde{q})$ a small neighborhood of the point $\tilde{q}$. It follows from the above that a suitable iterate of $\tilde{F}: (z,w) \mapsto (\tilde{f} - \delta w, z)$ maps $\mathcal{N}(\tilde{q})$ to a neighborhood of $(z_0, w_0)$. Notice that $\tilde{F}^2(z_0, w_0) = (z_1, w_1)$, hence a further suitable iterate of $\tilde{F}$ maps $\mathcal{N}(\tilde{q})$ to a neighborhood $\mathcal{N}(z_N, w_N)$.

It is clear from (ii) - (iv) that the size of $\mathcal{N}(z_N, w_N)$ can be chosen uniformly, i.e. the radius of the inner ball centered at $(z_N, w_N)$ does not shrink to zero as $\tilde{F}$ approximates $F$ better and better on the cylinder. Hence condition (i) implies that the local stable manifold of $\Sigma_{\tilde{F}}(\eta(\tilde{F})$ can be made to intersect $\mathcal{N}(z_N, w_N)$. It follows that the global stable manifold of $\tilde{F}$ passes through $\mathcal{N}(\tilde{q})$, which completes the proof.
\end{proof}

\section{Pseudoconvexity of the Fatou set}

The goal of this section is to prove pseudoconvexity for all Fatou components of transcendental H\'enon maps, from which it follows that the Julia set is perfect. We recall that there are several non-equivalent definitions of the Fatou set, depending on the used compactification of $\mathbb C^2$. We will prove pseudoconvexity of components using the $\mathbb P^2$-Fatou set, the preferred definition from \cite{henon1}.

We recall that the Julia set being perfect for the $\widehat{\mathbb C^2}$-Fatou set was already proved in \cite{FS98}. We do not know whether components of the $\widehat{\mathbb C^2}$-Fatou set are necessarily pseudoconvex. The result from \cite{FS98} in fact holds for all holomorphic automorphisms not conjugate to an upper triangular map with a repelling fixed point. Note that for the inverse of such a triangular map all orbits converge to this fixed point.

\medskip

From now on we only refer to the $\mathbb P^2$-Fatou set, and we will show that Fatou components are pseudoconvex for all holomorphic automorphisms whose inverse admits an escaping point. Note that by Proposition \ref{prop:two} this includes all transcendental H\'enon maps. We recall that pseudoconvexity of \emph{recurrent} Fatou components of holomorphic automorphisms was already proved in \cite{FS98}. Not surprisingly, most of our efforts will go into dealing with unbounded orbits.


Recall the following theorem by Levi \cite{Levi}, see also \cite{Ivashkovitch}.
\begin{thm}[Levi]\label{thm:Levi}
Let $H\subset \C^2$ be the standard Hartogs figure, and let $f$ be a meromorphic function on $H$. Then there exists a meromorphic function $F$ on $\hat H$ which extends $f$.
\end{thm}

Let us recall the following observation (Lemma 4.3) from~\cite{henon1}.
\begin{lem}\label{lem:boundary} Let $\Omega$ be a Fatou component. Then for any limit function $h$ on $\Omega$ either $h(\Omega)\subset \ell_\infty$ or $h(\Omega)\subset \C^2$.
\end{lem}

 \begin{theorem}\label{thm:Convexity of Fatou set}
Let $F$ be an automorphisms for which $F^{-1}$ has an escaping point.  Then every  Fatou component $\Omega$  of $F$ is pseudoconvex.
\end{theorem}
\begin{proof}
Let $H$ be a Hartogs figure relatively compact in $\Omega$.
We will show that the family $(F^n)$ is normal on the hull $\hat H$, and thus $\hat H\subset \Omega$.
If $(F^{n_k})$ is any subsequence of iterates, we will show that we can always extract a subsequence converging uniformly on compact sets to a holomorphic map $\psi\colon \hat H\to \P^2(\C)$.

If there exists $R>0$ so that $F^{n_k}(H)\subset  B(0,R)$ for all $k$, then by the maximum principle $F^{n_k}(\hat H)\subset   B(0,R)$ for all $k$ and thus $(F^{n_k})$ admits a convergent subsequence. On the other hand, if such $R$ does not exist, then there exists a sequence $(x_k)$ in $H$ and a subsequence $(n_k)$ such that
$\|F^{n_k}(x_k)\|\to +\infty$. Since $H$ is relatively compact in the Fatou set, up to extracting another subsequence we can assume that
$F^{n_k}|_H$ converges uniformly to a holomorphic map  $\varphi\colon H\to \mathbb{P}^2(\C)$. Clearly $\varphi(H)$ cannot be contained in $\C^2$, and thus by Lemma \ref{lem:boundary} it is contained in the line at infinity $\ell_\infty$.

We identify the line at infinity $\ell_\infty$ with the Riemann sphere $\hat{\C}$ via the map $[x:y:0]\mapsto y/x$.
With this identification we can see $\varphi$ as a meromorphic function on $H$ without indeterminacy points.
By Hartogs extension theorem for meromorphic functions (see Theorem \ref{thm:Levi}),
 there exists a meromorphic function $\psi:\hat{H}\rightarrow \hat\C$ which extends $\phi$.
We claim that
\begin{itemize}
\item[(A)] $\psi$ has no points of indeterminacy (and thus is a holomorphic map $\psi\colon \hat{H}\to \ell_\infty$) and
the sequence $\varphi_k$ converges to $\psi$ uniformly on compact subsets of $\hat H$,
\item[(B)] the sequence $(F^{n_k})$ converges uniformly on compact subsets of $\hat{H}$ to $\psi$.
\end{itemize}

Let us first show that  there exists an $R>0$ so that for all large enough $k$ there exists a point
$p_k\in B(0,R)$ so that $p_k\not\in F^{n_k}(\hat{H})$. Indeed, otherwise for any given $R$ there exists arbitrarily large $k(R)$ so that $ B(0,R)$ is contained in $F^{n_{k(R)}}(\hat{H}).$ By assumption there exists an escaping point $q$ for $F^{-1}.$ Choose $R>\|q\|$ so large that $\hat{H}\subset\subset   B(0,R).$
Then $F^{-n_{k(R)}}$ maps $ B(0,R)$ into the subset $\hat{H}.$
This implies that $F^{-n_{k(R)}}$ has an attracting fixed point in $\hat{H}.$
Moreover the basin of attraction contains $B(0,R)$. This contradicts the fact  that $q $ is an escaping
point for $F^{-1}$ and hence for $F^{-n_{k(R)}} $ and gives the existence of the desired $p_k$.

Let $\pi_k$ denote the radial projection from $\mathbb C^2\setminus p_k$ to $\ell_\infty.$
Then the functions
$$\varphi_k:=\pi_k\circ F^{n_k}$$ are holomorphic on $\hat{H}$, and the sequence $(\phi_k)$  converges
to $\phi$ uniformly on $H.$
\\

\noindent {\bf Proof of Claim (A)}. Up to changing holomorphic coordinates we can assume that $\hat H$ is the unit polydisk and that $H$ is the Euclidean Hartogs figure.
The function $\psi$
 has a zero set $Z$ and a pole set $P$ in $\hat{H}.$
Likewise the functions $\phi_k$ have zero sets $Z_k$ and pole sets $P_k$ in $\hat H$.
Since $\varphi_k$ is holomorphic on $\hat H$, the sets $Z_k$ and $P_k$ do not intersect for any $k.$

\begin{lemma}\label{familycurves}
Let $X$ be an analytic  complex curve in $\hat H$. For every compact set $K\subset \hat H\setminus (X\cup\{w=0\})$ there exists  a neighborhood $U$ of $X$  such that for all $y\in K$ there exists a complex analytic curve $S_y\subset \hat H$ containing $y$ which intersects $U$ only inside $H$
and whose boundary $\partial S_y$ is contained in $\{|z|=1\}\subset \partial H$.
\end{lemma}

\begin{proof}[Proof of Lemma \ref{familycurves}]
Let $f\colon \hat H\to \C$ be a holomorphic function such that $X=\{f=0\}$.
Up to considering a smaller Hartogs figure we can assume that $f$ is bounded and, up to multiplication by a non-zero constant, that $|f|<1$ on $\hat{H}.$ Let us define the family of functions $f_{C,n}:\Hh\ra\C$ by
 $$
 f_{C,n}:=f+Cw^n, \quad n\in \N, C\in \C.
 $$

Let $y=(z_0,w_0)$ be a point in $K$.
For every $n$ the point  $y$ is in the zero locus of the function $f_{C,n}$ for the constant $C_{n,y}=\frac{-f(z_0,w_0)}{w_0^n}\ra\infty$.
If $n$ is large enough, then $|f(z,w)+C_{n,y} w^n|>1$ for $|w|=1$, $y\in K$.
This implies that for all $y\in K$ the curve  $$S_y:=\{   f+C_{n,y}w^n=0\} $$ does not intersect the part of the boundary of $\Hh$ at which $ |w|=1$. Therefore we have that $\partial S_y$ is contained in $\{|z|=1\}\subset \partial H$.
Notice that any curve $ S_y$ intersects $X$ only in $\{w=0\}$. Hence there exists a neighborhood $U$ of $X$ such that
$S_y\cap U\subset H$. By compactness of $K$ and the continuity of the family of curves $S_y$ we obtain the result.
\end{proof}


Let $K\subset \hat{H}\setminus (P \cup \{w=0\})$ be a compact subset. Let $U$ be the neighborhood given by  Lemma \ref{familycurves} with $X=P$. There exists a constant $C>0$ such that $\{|\psi|>C\}$ is contained in $U$. By the maximum principle
applied to the sets $S_y\setminus U$,
the sequence  $\varphi_k\colon \hat H\to \hat \C$ is uniformly bounded on $K$ by Lemma \ref{familycurves}.
Hence by Vitali's Theorem the sequence  $(\varphi_k)$ converges to $\psi$ uniformly on compact subsets of  $\hat{H}\setminus P$. Arguing similarly we obtain convergence on compact subsets of  $\hat{H}\setminus Z$, and thus on the complement of indeterminacy points of $\psi$.

To finish the proof of  Claim (A), we need to prove that there are no indeterminacy points.

Suppose that there is at least one indeterminacy point.
By a linear change of coordinates  we can assume that the indeterminacy point is at the origin,
and that there exists a small $\delta_2>0$ such that $Z$ and $P$ intersect  the vertical disc $\{0\}\times \D(0,\delta_2)$ only at the origin.
Moreover, there exists a small $\delta_1>0$ such that the intersections of $Z$ and $P$ with the bidisk $\Delta:=\Delta(0, (\delta_1,\delta_2))$ project as a proper map to $\D(0,\delta_1)$.
Up to further decreasing $\delta_1$ and $\delta_2$ if necessary, we can assume that the origin is the only indeterminacy point of $\psi$.
Since $\varphi_k$ converges to $\psi$ uniformly on $\Delta\setminus\{0\}$ the analytic sets $Z_k$ and $P_k$ cannot intersect the set $|w|=\delta_2$, and thus they also project as a proper map to $\D(0,\delta_1)$.
It follows that $Z,P, Z_k,P_k$ can be considered as branched coverings over the $z$-variable
with finitely many branch points in $D(0,\delta_1)$.

The intersections of $Z$ and $P$ with the bidisk  $\Delta$
 can be written, counting multiplicities, as
 $$Z\cap \Delta=\{w=\alpha_m(z)\colon m=1,\dots ,M\} ,$$
$$P\cap \Delta=\{w=\beta_n(z)\colon n=1,\dots, N\},$$
where $\alpha_m$ and $\beta_n$ are multifunctions from $\D(0,\delta_1)$ to $\D(0,\delta_2)$.

We claim that  for large enough $k$  the intersections of $Z_k$ and $P_k$ with the bidisk  $\Delta$ can be written similarly, using functions $\alpha_m^k,\beta_n^k\colon \D(0,\delta_1)\to \D(0,\delta_2)$, with the same index sets:
$$Z_k\cap \Delta=\{w=\alpha^k_m(z)\colon m=1,\dots ,M\} ,$$
$$P_k\cap \Delta=\{w=\beta^k_n(z)\colon n=1,\dots, N\}.$$
Indeed, let $z_0$ be a point in the punctured disc $\D(0,\delta_1)\setminus\{0\}$. Then $w\mapsto \varphi_k(z_0,w)$ converges uniformly to $w\mapsto \psi(z_0,w)$ on $\overline\D(0,\delta_2)$, and thus for large $k$ they have the same number of zeros and poles in the disc $ \D(0,\delta_2)$. By continuity, the number of zeros and poles in the disc $\D(0,\delta_2)$  of the  map $w\mapsto \varphi_k(z,w)$ does not depend on $z$, which proves the claim.

Consider the  functions  $$h(z):=\prod_{\substack{n=1,\dots,N \\m= 1,\dots, M}}(\beta_n(z)-\alpha_m(z)),\quad h_k(z):=\prod_{\substack{n=1,\dots,N \\m= 1,\dots, M}}(\beta^k_n(z)-\alpha^k_m(z)),$$ which are  holomorphic functions defined on $\D(0,\delta_1)$ in the complement of the branch points. They extend holomorphically across the branch points.
Notice that $h$ has a zero only at the origin, while every $h_k$ is zero-free.
Since $g_k\to g$ uniformly on compact subsets of $\B(0,\epsilon)\setminus \{0\}$, it follows that
$h_k\to h$ uniformly on compact subsets of $\D(0,\delta_1)\setminus \{0\}$.
But since $h$ is defined and holomorphic at the origin, it follows that $h_k\to h$ uniformly on $ \D(0,\delta_1)$.
But this contradicts Hurwitz's theorem.
\\

\noindent {\bf Proof of Claim (B)}. Since the sequence $\varphi_k=\pi_k\circ F^{n_k}$ converges to $\psi$ uniformly on compact subsets of $\hat H$,
Claim (B)  immediately follows once we prove that $F^{n_k}\ra\ell_\infty$ as $k\ra\infty$ uniformly on compact subsets of $\hat{H}$. Let us define  $G_k:=F^{n_k}-p_k=(A_k,B_k)$ for some holomorphic functions $A_k, B_k:\C^2\ra\C$.
Notice that $\psi=\lim_{k\ra\infty} \pi_0\circ G_k$, and that $F^{n_k}\ra\ell_\infty$  if and only if $G_{k}\ra\ell_\infty$, since all $p_k$ are contained in a bounded set.
Our choice of   identification $\ell_\infty\simeq \hat \C$ gives $\psi=\lim_{k\ra\infty}\frac{B_k}{A_k}$. Let $Z$ and $P$ denote the zeros and the poles of   $\psi$ respectively. Since $\psi$ has no indeterminacy points, $Z$ and $P$ are disjoint.

Let $K\subset \hat{H}\setminus (Z \cup \{w=0\})$ be a compact subset. Let $U$ be the neighborhood given by  Lemma \ref{familycurves} with $X=Z$. There exists a constant $C>0$ such that $\{|\psi|<C\}$ is contained in $U$.
Let $y\in K$ and let $S_y$ be the complex analytic curve given by Lemma \ref{familycurves}.
Recall that on $\bar H$ we have that $\|G_k\|\to \infty$, and notice that on
$\bar H\setminus U$ we have $\frac{B_k}{A_k}\leq \frac{C}{2}$ for big enough $k$. Hence $B_k\to \infty$ uniformly on $\bar H\setminus U$. By the maximum principle applied to $\frac{1}{B_k}$ on the set $S_y\setminus U$ we obtain that $B_k\to \infty$ on $y$, uniformly over all $y \in K$.
One can argue similarly for a compact set $K\subset \hat{H}\setminus (P \cup \{w=0\})$.

\end{proof}

Since the inverse of transcendental Henon maps are also transcendental Henon maps and have an escaping point by Theorem~\ref{eremenkotheorem}, we have the following corollary.
\begin{corollary}
Let $F$ be a transcendental Henon map. Then every Fatou component $U$ is pseudoconvex, and thus the Julia set of $F$ has no isolated points.
\end{corollary}

\bibliographystyle{amsalpha}

\bibliography{Henon2}

\end{document}